\documentclass[11pt,oneside]{amsart}

\usepackage{amssymb}
\usepackage{amsmath}
\usepackage{amsthm}
\usepackage[all]{xy}

\theoremstyle{plain}
\newtheorem{thm}{Theorem}[section]
\newtheorem{prop}[thm]{Proposition}
\newtheorem{lem}[thm]{Lemma}
\newtheorem{cor}[thm]{Corollary}

\newtheorem{clm}[thm]{Claim}

\theoremstyle{definition}
\newtheorem{defn}[thm]{Definition}

\theoremstyle{remark}
\newtheorem{rem}[thm]{Remark}

\DeclareMathOperator{\mult}{mult}

\DeclareMathOperator{\Proj}{Proj}
\DeclareMathOperator{\codim}{codim}

\DeclareMathOperator{\Cone}{Cone}
\DeclareMathOperator{\Hilb}{Hilb}
\DeclareMathOperator{\ord}{ord}
\DeclareMathOperator{\Hom}{Hom}

\def\N{\mathbb{N}}

\def\R{\mathbb{R}}
\def\C{\mathbb{C}}

\def\r+{\mathbb{R}_{\geq 0}}

\def\ep{\varepsilon}

\def\r+{{\R}_{\geq 0}}
\def\P{\mathbb{P}}
\def\arw{\rightarrow}
\def\*c{\C^{\times}}

\newcommand{\calo}{\mathcal {O}}

\newcommand{\m}{\mathfrak{m}}

\begin{document}

\title{Seshadri constants and degrees of defining polynomials}
\author{Atsushi Ito} 
\address{Graduate School of Mathematical Sciences, 
The University of Tokyo, 3-8-1 Komaba, 
Meguro, Tokyo, 153-8914, Japan.}
\email{itoatsu@ms.u-tokyo.ac.jp}

\address{Graduate School of Mathematical Sciences, 
The University of Tokyo, 3-8-1 Komaba, 
Meguro, Tokyo, 153-8914, Japan.}
\email{muon@ms.u-tokyo.ac.jp}
\author{Makoto Miura}

\begin{abstract}
In this paper,
we study a relation between Seshadri constants and
degrees of defining polynomials.
In particular,
we compute the Seshadri constants on
Fano varieties obtained as complete intersections
in rational homogeneous spaces of Picard number one.
\end{abstract}

\subjclass[2010]{14C20}
\keywords{Seshadri constant, defining polynomial, rational homogeneous space}
\dedicatory{Dedicated to Professor~Yujiro~Kawamata on the~occasion of his~sixtieth~birthday.}

\maketitle


\section{Introduction}
Seshadri constant was introduced by Demailly in \cite{Dem},
as an  invariant which measures the local positivity of ample line bundles.

\begin{defn}\label{def of sc_at a point}
Let $L$ be an ample line bundle on a projective variety $X$,
and take a (possibly singular) closed point $p \in X$. 
The \textit{Seshadri constant} of $L$ at $p$ is defined to be
\[
\ep(X,L;p) := \max\{ \, t \geq 0 \, | \, \mu^* L-tE \ \text{is nef} \, \},
\]
where $\mu: \widetilde{X} \arw X$ is the blowing up at $p$ and $E= \mu^{-1}(p)$ is the exceptional divisor.

Equivalently,
the Seshadri constant can be also defined as
\[
\ep(X,L;p)=\inf_C \left\{\dfrac{C.L}{\mult_p(C)}\right \} ,
\]
where the infimum is taken over all reduced and irreducible curves $C$ on $X$ passing through $p$.
We call a curve $C$ a \textit{Seshadri curve} of $L$ at $p$
if $\ep(X,L;p)= C.L / \mult_p (C)$.
\end{defn}

Seshadri constants have many interesting properties (see \cite[Chapter 5]{La} for instance),
but it is very difficult to compute them in general.
Many authors study surface cases,
but computations in higher dimensions are rare.

Let $X$ be a projective variety embedded in $\P^N$ and $p \in X$.
The purpose of this paper is to study the Seshadri constant $\ep(X,\calo_X(1);p)$
by investigating homogeneous polynomials which define $X$.
It is easy to see that $\ep(X,\calo_X(1);p) \geq 1$ for such $X$ and $p$.
Furthermore it is known that
$\ep(X,\calo_X(1);p)=1$ holds if and only if there exists a line on $X$ passing through $p$ (cf.\ \cite[Lemma 2.2]{Ch}).

In \cite{Ba} and \cite{Ch},
Bauer and Chan give a lower bound of  $\ep(X,\calo_X(1);p)$
by using the degree $\deg(X):=\calo_X(1)^{\dim X}$
when $X$ is a surface and a 3-fold respectively.

\begin{thm}[cf.\ {\cite[Theorem 2.1]{Ba}}, {\cite[Theorem 1.4]{Ch}}]\label{Bauer and Chen}
Let $X$ be a smooth projective surface or a $3$-fold in $\P^N$, and $p \in X$ a point.
If there exists no line on $X$ passing through $p$,
it holds that
\[
\ep(X,\calo_X(1);p) \geq \frac{\mathrm{deg}(X)}{\mathrm{deg}(X)-1} .
\]
\end{thm}

\begin{rem}
Bauer proved the sharpness of the lower bound as well,
that is,
for each $d \geq 3$,
there exist a smooth surface $X$ and $p \in X$
such that $d=\deg (X)$ and $\ep(X,\calo_X(1);p) = d/ (d-1)$.
Chan also constructed such $X$ and $p$ in $3$-dimensional case for $d \geq 4$,
although $X$ might have finitely many singular points.
\end{rem}

Instead of the degree,
we introduce $d_p(X)$
for a projective variety $X \subset \P^N$ and $p \in X$ (cf.\ \cite[Definition 1.8.37]{La}).

\begin{defn}\label{def_deg}
Let $X$ be a projective variety in $\P^N$
and $p \in X$ a point.
We define $d_p(X)$ to be the least integer $d$
such that the natural map
\[
H^0(\P^N,I_X \otimes \calo_{\P^N}(d)) \otimes \calo_{\P^N} \arw I_X \otimes \calo_{\P^N}(d)
\]
is surjective at $p$,
where $I_X \subset \calo_{\P^N}$ is the ideal sheaf corresponding to $X$.
In other words,
$X$ is cut out scheme theoretically by hypersurfaces of degree $d_p(X)$ at $p$.
\end{defn}

By using $d_p(X)$,
we give a lower bound of the Seshadri constant on $X \subset \P^N$ (which may be singular) in any dimensions.

\begin{thm}\label{intro thm2}
Let $X$ be a projective variety in $\P^N$ and $p \in X$ a point.
If there exists no line on $X$ passing through $p$,
it holds that
\[
\ep(X,\calo_X(1);p) \geq \frac{d_p(X)}{d_p(X)-1} .
\]

Furthermore,
this lower bound is sharp,
i.e.,
for any $n \geq 1$ and $d \geq 2$,
there exist a smooth projective variety $X \subset \P^N$ and $p \in X$
such that $n=\dim X$, $d=d_p(X)$, and $\ep(X,\calo_X(1);p) = d/(d-1)$.
\end{thm}

\begin{rem}
It holds $d_p(X) \leq \deg(X)$ for any projective variety $X \subset \P^N$ and $p \in X$
(see the proof of \cite[Theorem 1]{Mu}).
Thus Theorem \ref{intro thm2} improves Theorem \ref{Bauer and Chen} even in the cases $\dim X=2,3$.
\end{rem}

In Section \ref{section computation},
we compute the Seshadri constants on some varieties $X \subset \P^N$
by finding a curve $C$ such that $\deg (C) / \mult_p (C) $ coincides with the lower bound $d_p(X) /(d_p(X)-1) $.
As a special case,
we obtain the following theorem.

\begin{thm}[=Corollary \ref{rational_homog}]\label{intro_thm1}
Let $Y \subset \P^N$ be a rational homogeneous space of Picard number $1$,
which is embedded by the ample generator.
Let $X$ be a complete intersection variety in $Y$ of hypersurfaces of degrees $d_1 \leq \ldots \leq d_r$
such that $-K_X=\calo_X(1)$
and $p \in X$.
If there exists no line on $X$ passing through $p$,
it holds that
\begin{align*}
\ep (X,\calo_X(1); p) =  \frac{d_p(X)}{d_p(X)-1} = \left\{ 
\begin{array}{cl} 
d_r/(d_r-1) & \text{when } d_r \geq 2 ,\\ 
2 & \text{when }  d_r =1 .\\
\end{array} \right.
\end{align*}
\end{thm}

\begin{rem}\label{rem for thm 1}
If $X$ is a complete intersection variety in $Y$ of hypersurfaces of degrees $d_1 \leq \ldots \leq d_r$
such that $-K_X=\calo_X(i)$ for $i \geq 2$,
it is easy to check that $X$ is covered by lines (see Lemma \ref{cone of F_p(X)}).
Thus $\ep(X,\calo_X(1);p) = 1$ holds for any $p \in X$.

Hence Theorem \ref{intro_thm1} states that
we can compute $\ep(X,\calo_X(1);p)$
for any Fano variety $X$ obtained as a complete intersection in any rational homogeneous space of Picard number $1$.
\end{rem}

Throughout this paper,
all schemes are defined over the complex number field $\C$.
In Section \ref{section lower bound}, we give the proof of Theorem \ref{intro thm2}.
In Section \ref{section computation}, we compute Seshadri constants on some varieties.

\subsection*{Acknowledgments}
The authors would like to express their gratitude to Professor Yujiro Kawamata
for his valuable advice, comments, and warm encouragement.
They are also grateful to Professors Katsuhisa Furukawa and Kiwamu Watanabe
for their useful comments and suggestions.


\section{Lower bounds}\label{section lower bound}

In this paper,
a \textit{line} means a projective curve of degree $1$ in $\P^N$.
The moduli of lines plays an important role in this paper.

\begin{defn}
Let $X$ be a projective scheme in $\P^N$
and $p \in X$ a point.
We denote by $F_p(X)$ the moduli space of lines on $X$ passing through $p$.
Note that $F_p(X)$ is naturally embedded in $F_p(\P^N) \cong \P^{N-1}$.
\end{defn}

For a graded ring $S$,
we denote by $S_i$ the set of all homogeneous elements of degree $i$.
We will use the following lemma in Sections \ref{section lower bound}, \ref{section computation}.

\begin{lem}\label{cone of F_p(X)}
Let $X$ be a projective scheme in $\P^N$
and $p \in X$ a point.
Fix homogeneous coordinates $x_0,\ldots,x_N$ on $\P^N$
such that $p=[1:0:\cdots:0]$.
Assume that $X$ is defined by homogeneous polynomials $\{f_j\}_{1 \leq j \leq r}$ around $p$,
and write $f_j=\sum_{i=1}^{d_j} x_0^{d_j-i} f_j^i $ for $d_j=\deg f_j$ and $f_j^i \in \C[x_1,\ldots,x_N]_i$.
Then $F_p(X) \subset F_p(\P^N) \cong \Proj \C[x_1,\ldots,x_N]$ is defined by $\{ f_j^i \}_{1 \leq j \leq r,1 \leq i \leq d_j}$.
\end{lem}

\begin{proof}
The proof is similar to that of Proposition 2.13 (a) in \cite{Deb}.
We leave the details to the reader.
\end{proof}

For a variety $X$, a point $p \in X$, and an effective Cartier divisor $D$,
we define $\ord_p(D):=\max \{ \, m \in \N \, | \, f  \in \m_p^m \}$,
where $f$ is a defining function of $D$ at $p$.
The following lemma is easy and well known,
but we prove it for the convenience of the reader.

\begin{lem}\label{local intersection number}
Let $C $ be a curve on a projective variety $X$ and $p \in C$ a point.
For an effective Cartier divisor $D$ on $X$ not containing $C$,
it holds
\[
C.D \geq \ord_p (D ) \cdot \mult_p (C) .
\]
\end{lem}

\begin{proof}
Let $\nu : \widetilde{C} \arw C$ be the normalization.
Then there exists an effective divisor $E $ on $\widetilde{C}$
such that $ \calo_{\widetilde{C}}(-E)=\nu^{-1}   \m_{C,p} $ and $ \deg (E)=\mult_p (C)$.
Since $\calo_C (- D |_C) \subset  \m_{C,p}^{\ord_p (D |_C)} $,
we have
\[
\calo_{\widetilde{C}}( - \nu^* D |_C) \subset  \calo_{\widetilde{C}}(- \ord_p (D |_C) E).
\]
This means
\begin{align*}
D.C=\deg (\nu^* D |_C) &\geq  \deg (\ord_p (D |_C) E) \\
&= \ord_p (D |_C) \cdot  \mult_p (C). 
\end{align*}
Since $\ord_p(D |_C) \geq \ord_p(D)$,
this lemma is proved.
\end{proof}

Now we can prove Theorem \ref{intro thm2}.
The idea is simple.
For any curve $C$ on $X$ passing through $p$,
we find a suitable divisor $D \in |\calo_X(i)|$ not containing $C$ for some $i$
and apply Lemma \ref{local intersection number}.

\begin{proof}[\bf{Proof of Theorem} \ref{intro thm2}]
Let $x_0,\ldots,x_N$ be homogeneous coordinates on $\P^N$
such that $p=[1:0:\cdots:0]$,
and set $d=d_p(X)$.
Choose and fix a basis $f_1,\ldots,f_r$
of $ H^0(\P^N,I_X \otimes \calo_{\P^N}(d) ) \subset \C[x_0,\ldots,x_N]_d$.
As in Lemma \ref{cone of F_p(X)},
we can write
\[
f_j=x_0^{d-1}f_j^1 + x_0^{d-2}f_j^2 + \cdots + x_0 f_j^{d-1} + f_j^d
\]
for some $f_j^i \in \C[x_1,\ldots,x_N]_i$.

For $1 \leq i \leq d-1, \, 1 \leq j \leq r$,
we set
\[
D_j^i = (x_0^{i-1} f_j^1 + \cdots + x_0 f_j^{i-1} + f_j^i=0) \subset \P^N.
\]
Note that $D_j^i$ can be $\P^N$ itself.
By the definition of $D_j^i$,
it holds
\begin{align}
\tag{$*$}
\begin{split}
X \ \cap  \bigcap_{1 \leq i \leq d-1, 1 \leq j \leq r }  D_j^i \ \
&\subset \ \ \bigcap_{1 \leq j \leq r} (f_j=0) \ \cap  \bigcap_{1 \leq i \leq d-1, 1 \leq j \leq r }  D_j^i \\
&=  \ \ \bigcap_{1 \leq i \leq d, 1 \leq j \leq r } (f_j^i = 0).
\end{split}
\end{align}
By the definition of $d_p(X)$,
we have $X= (f_1 = \cdots = f_r=0)$ around $p$.
Hence
the last term of $(*)$ is nothing but $\Cone F_p(X)$ by Lemma \ref{cone of F_p(X)},
where $\Cone F_p(X)$ is the projective cone of $F_p(X)$ in $\P^N$ with the vertex at $p$.
Since $F_p(X)=\emptyset$ by assumption,
we have
\begin{align}
\tag{$\dagger$} \bigcap_{1 \leq i \leq d-1, 1 \leq j \leq r }  D_j^i |_X  = \{ p \}.
\end{align}

Fix a curve $C \subset X$  passing through $p$.
To show the inequality in this theorem,
it suffices to show
\[
\frac{\deg (C)}{\mult_p (C)} \geq \frac{d}{d-1} . 
\]
By $(\dagger)$,
there exist $1 \leq i \leq d-1, \, 1 \leq j \leq r  $
such that $D_j^i |_X$ does not contain $C$.
In particular,
$D_j^i |_X $ is an effective divisor on $X$ not containing $C$.
Since $x_0^{d-i}(x_0^{i-1} f_j^1 + \cdots + x_0 f_j^{i-1} + f_j^i) = -(x_0^{d-i-1}f_j^{i+1}  \cdots + x_0 f_j^{d-1} + f_j^d )$ on $X$,
it holds that
\begin{align*}
\ord_p (D_j^i |_X) &= \ord_p \left( \frac{x_0^{i-1} f_j^1 + \cdots + x_0 f_j^{i-1} + f_j^i}{x_0^i} \, \bigg|_X \right) \\
                        &= \ord_p \left( - \frac{x_0^{d-i-1}f_j^{i+1}  \cdots + x_0 f_j^{d-1} + f_j^d}{x_0^d} \, \bigg|_X \right) \geq i+1.
\end{align*}
Since $D_j^i |_X \sim \calo_X(i)$,
we have
\[
 i \, \deg(C) = (D_j^i |_X) .C \geq \ord_p(D_j^i |_X) \cdot \mult_p(C) \geq (i+1) \mult_p(C)
\]
by Lemma \ref{local intersection number}.
Hence 
it holds $ \deg (C) / \mult_p (C) \geq (i +1) / i \geq d/(d-1)  $ by $1 \leq i \leq d-1$,
and the inequality of this theorem is shown.

\vspace{5mm}
Now, we show the sharpness of the lower bound.
First, assume $d \geq n+1$.
Set
\[
f=x_0^{d-1}f^1 + \cdots + x_0^{d-n+1} f^{n-1} + x_0 f^{d-1} + f^d \in \C[x_0,\ldots,x_{n+1}]
\]
for a general $f^i \in \C[x_1,\ldots,x_{n+1}]_i$ for each $1 \leq i \leq n-1$ and $i=d-1,d$.
Then $f$ defines a smooth hypersurface $X \subset \P^{n+1}$ containing $p=[1:0:\cdots:0]$.
By the generality of $f^i$,
\[
F_p(X) =( f^1=\cdots=f^{n-1}=f^{d-1} = f^d=0) \subset \Proj \C[x_1,\ldots,x_{n+1}]
\]
is empty.
Set 
\[
C= (f^1=\cdots=f^{n-1}=x_0 f^{d-1} + f^d=0) \subset \P^{n+1}.
\]
By definition,
$C$ is contained in $X$ and contains $p$.
Since all $f^i$ are general,
$C$ is a complete intersection curve.
Hence
\[
\deg(C)=(n-1)! \cdot d, \quad \mult_p (C)=(n-1)! \cdot (d-1)
\]
hold, and we have
\[
\ep(X,\calo_X(1);p) \leq \deg(C) / \mult_p(C) =d / (d-1).
\]
Since $d_p(X)=d$ and $F_p(X)= \emptyset$,
it holds $\ep(X,\calo_X(1);p) =d / (d-1) $ by the inequality of this theorem.

When $d \leq n$, we use Theorem \ref{intro_thm1},
which is proved in the next section.
(Of course, we do not use the sharpness assertion in Theorem \ref{intro thm2}
to show Theorem \ref{intro_thm1}.)
For $2 \leq d \leq n$,
choose positive integers $r$ and $d_1 \leq \cdots \leq d_r$
such that $d=d_r$ and $\sum_{j=1}^r d_j =n+r$.

Applying Theorem \ref{intro_thm1} to $Y=\P^{n+r}$,
we have $ \ep (X,\calo_X(1); p)=d_r/(d_r-1) =d/(d-1) $
for a smooth complete intersection $X$ of hypersurfaces of degrees $d_1 \leq \ldots \leq d_r$
and general $p \in X$.
Note that we can easily check $F_p(X) = \emptyset$ for general $p \in X$ by using Lemma \ref{cone of F_p(X)}.
Since $d_p(X)=d_r=d$,
the sharpness is proved. 
\end{proof}

\section{Finding Seshadri curves}\label{section computation}

In \cite{It},
the first author computes Seshadri constants on some Fano manifolds
at a very general point.
In the paper,
toric degenerations are used to estimate Seshadri constants from below.
Since the lower semicontinuity is used there for lower bounds,
we have to assume some very generality in the method.
Instead of toric degenerations,
we use the lower bound in Theorem \ref{intro thm2} here.
For upper bounds,
we find Seshadri curves
similar to \cite{It}.
That is,
for some variety $X$,
we can find a curve $C \subset X$ passing through $p$
such that $\deg(C)/ \mult_p(C) $ coincides with the lower bound in Theorem \ref{intro thm2}.

To show Theorem \ref{intro_thm1},
we treat the case $d_r \geq 2$ in Subsection \ref{subsection cutting}.
In that case,
we construct a Seshadri curve 
by cutting a suitable cone in $X$ by hypersurfaces.
In Subsection \ref{subsection finding conic},
we treat the case $d_r = 1$.
In that case,
we show the existence of a conic $C \subset X$ passing through $p$
by using the deformation theory. 
In Subsection \ref{subsection_proof_of_cor},
we prove Theorem \ref{intro_thm1}.

\subsection{Cutting cones by hypersurfaces}\label{subsection cutting}

In this subsection,
we prove the following theorem,
from which the case $d_r \geq 2$ in Theorem \ref{intro_thm1} follows immediately.

\begin{thm}\label{main_thm}
Let $Y$ be a projective variety in $\P^N$
and $p \in Y$ a point.
For a subvariety $X \subset Y$ containing $p$,
we assume the following:
\begin{itemize}
\item[i)] Around $p$, $X$ is a locally complete intersection in $Y$ of hypersurfaces of degrees $d_1 \leq \ldots \leq d_r$.
That is, $r=\codim (X,Y)$ and
there exists $f_j \in H^0(\P^N,\calo_{\P^N}(d_j))$ for each $j$
such that
\[
X=Y \cap \bigcap_{1 \leq j \leq r} (f_j=0)
\]
holds in a neighborhood of $p$.
\item[ii)] $F_p(Y) \not = \emptyset$ and $\sum_{j=1}^r d_j \leq \dim F_p(Y) +1$.
\item[iii)] $d_p(Y) \leq d_r$.
\end{itemize}
Then it holds that
\begin{align*}
\ep (X,\calo_X(1); p) = \left\{ 
\begin{array}{cl} 
 1 &  \text{when }F_p(X)\ne \emptyset,\\
d_r/(d_r-1)& \text{when } F_p(X)= \emptyset.\\
\end{array} \right.
\end{align*}
\end{thm}

\begin{proof}
When $F_p(X) \not = \emptyset$,
this theorem is clear.
Thus we may assume $F_p(X)= \emptyset$,
and show $\ep(X,\calo_X(1);p) = d_r/(d_r-1)$.
As in the proof of Theorem \ref{intro thm2},
we take homogeneous coordinates $x_0,\ldots,x_N$ on $\P^N$
such that $p=[1:0:\cdots:0]$,
and write $f_j=\sum_{i=1}^{d_j} x_0^{d_j-i} f_j^i$
for $f_j^i \in \C[x_1,\ldots,x_N]_i$.
By Lemma \ref{cone of F_p(X)},
we have
\[
F_p(X)=F_p(Y) \, \cap \ \bigcap_{i,j}  \ (f_j^i=0) \subset \Proj \C[x_1,\cdots,x_N] .
\]
Hence
$F_p(X)$ is an intersection of $F_p(Y)$ and $\sum d_j$ hypersurfaces.
Thus
$\sum_{j=1}^r d_j =\dim F_p(Y) +1$ must hold by the condition ii) and $F_p(X)= \emptyset$.

By Theorem \ref{intro thm2},
we have $\ep(X,\calo_X(1);p) \geq d_r/(d_r-1)$ since $d_p(X) \leq \max \{ d_p(Y),d_r\} =d_r$.
To show the opposite inequality,
we may assume that each $f_j^i$ is general
because of the lower semicontinuity of Seshadri constants (cf.\ \cite[Example 5.1.11]{La}).
For general $f_j^i$,
we can find a curve $C \subset X$ through $p$
such that $\deg (C) / \mult_p (C) = d_r/(d_r-1)$ as follows.

Fix an irreducible component $Z$ of $F_p(Y)$ such that $\dim Z=\dim F_p(Y)$.
We define a curve $C \subset \P^N$ to be
\begin{align*}
C= \Cone Z &\cap \bigcap_{1 \leq j \leq r-1}  (f_j^1=\cdots =f_j^{d_j-1}=f_j^{d_j} =0) \\
&\cap (f_r^1=\cdots=f_r^{d_r-2} =x_0 f_r^{d_r-1}+f_r^{d_r} =0).
\end{align*}
Note $d_r \geq 2$ holds because $X$ is not linear and $d_p(X) \leq d_r$.
Since
\[
\Cone Z \subset \Cone F_p(Y) \subset Y,
\]
$C$ is contained in $X$.
By definition, $C$ is cut out from $ \Cone Z$ by $\sum_{i=1}^{r-1} d_j + (d_r-1)=\sum_{j=1}^r d_j  -1$ hypersurfaces,
and $ \sum_{j=1}^r d_j =\dim F_p(Y) +1 =\dim \Cone Z$.
Since all $f_j^i$ are general,
$C$ is a reduced and irreducible curve.
By definition, we have
\begin{align*}
\deg (C) &= \deg (\Cone Z)  \cdot d_1 ! \cdots d_{r-1} ! \cdot (d_r-2)! \cdot d_r ,\\
\mult_p (C) &= \mult_p (\Cone Z) \cdot d_1 ! \cdots d_{r-1} ! \cdot (d_r-2)! \cdot (d_r -1).
\end{align*}
Thus it holds that $\deg (C) / \mult_p (C) =d_r / (d_r-1) $
because $ \deg (\Cone Z) = \deg (Z)  = \mult_p (\Cone Z)$.
Hence $\ep(X,\calo_X(1);p) \leq \deg (C) / \mult_p (C) = d_r/(d_r-1)$ holds
and this theorem follows.
\end{proof}

\subsection{Finding conics}\label{subsection finding conic}

When $Y \subset \P^N$ is a rational homogeneous space of Picard number $1$
other than a projective space,
$d_p(Y)=2$ for any $p \in Y$.
Hence we cannot apply Theorem \ref{main_thm} to the case
when $d_r=1$,
i.e.,
$X$ is a section of $Y$ by hyperplanes.

Since $d_p(X)=d_p(Y)=2$ holds for such $X$,
the lower bound obtained by Theorem \ref{intro thm2} is $d_p(X) / (d_p(X)-1)=2$.
Thus if there exists a (smooth) conic $C \subset X$ passing through $p$,
we have
\[
2 \leq \ep(X,\calo_X(1);p) \leq \frac{\deg (C)}{\mult_p(C)} =2 .
\]

\vspace{2mm}
For the following proposition,
we prepare some notations.
For a subvariety $X$ in a variety $Y$,
we denote by $I_{X/Y} $ the ideal sheaf on $Y$ corresponding to $X$.
A \textit{conic} in $\P^N$ is a smooth projective curve of degree $2$,
and a \textit{plane} in $\P^N$ is a $2$-dimensional linear projective subspace.
For a projective variety $Y \subset \P^N$,
we say that $Y$ is \textit{covered by} lines (resp.\ conics, planes)
if for general $p \in Y $,
there exists a line (resp.\ a conic, a plane) on $Y$ containing $p$.

\begin{prop}\label{existence_of_conic}
Let $Y \subset \P^N$ be a smooth projective variety satisfying the following:
\begin{itemize}
\item[i)] $I_{Y/\P^N} \otimes \calo_{\P^N}(2)$ is globally generated,
\item[ii)] for a general $p\in Y$,
$\dim R_p(Y) =2$ holds,
where $R_p(Y)$ is the subscheme of the Hilbert scheme $\Hilb(Y)$
which parametrizes conics on $Y$ passing through $p$.
\end{itemize}

Then for general $p \in Y$ and general hyperplane $H \subset \P^N$ containing $p$,
there exists a conic $C$ on $Y \cap H$ passing through $p$. 
\end{prop}

\begin{proof}
Fix a general point $p \in Y$.
Since $\dim R_p(Y)=2$,
we can choose an irreducible component $\mathcal R$ of $R_p(Y)$ of dimension $2$.
We define the incidence variety $I$ as
\[
 I= \{ (C,H) \in \mathcal R \times |\calo_{\P^N}(1) \otimes \m_p| \, | \, C \subset H \},
\]
and consider the natural projections
\[
\xymatrix{
&I \ar[dl]_{\pi_1}  \ar[dr]^{\pi_2} &\\
\mathcal R & &  |\calo_{\P^N}(1) \otimes \m_p| .  \\
}\]

To show this proposition,
we have to show that the projection $\pi_2$ is generically surjective.
For a fixed $C \in \mathcal R$,
a hyperplane $H \in |\calo_{\P^N}(1) \otimes \m_p|$ contains $C$
if and only if $H$ contains the plane spanned by $C$.
Thus $\pi_1$ is a $\P^{N-3}$-bundle
and it holds that
\[
\dim I=\dim \mathcal R + N-3 =N-1=\dim  |\calo_{\P^N}(1) \otimes \m_p|.
\]
Hence it suffices to show that
$\dim \pi_2^{-1} (H)=0$ for a general $H \in \pi_2(I)$.
\vspace{2mm}
\ \\
\textbf{Step 1.} 
In this step,
we show the following claim.

\begin{clm}\label{glob_gen_of_ideal}
In the above setting,
the composition of the natural maps
\begin{align*}
H^0(\P^N,I_{C/\P^N} \otimes \calo_{\P^N}(1)) \otimes \calo_{\P^N} &\arw
H^0(Y, I_{C/Y} \otimes \calo_Y(1)) \otimes \calo_Y \\
&\arw  I_{C/Y} \otimes \calo_Y(1)
\end{align*}
is surjective for $C \in \mathcal R$.
\end{clm}

\begin{proof}[Proof of Claim \ref{glob_gen_of_ideal}]
Fix $C \in \mathcal R$, and let $P_C \subset \P^N$ be the plane spanned by $C$.
Since $\dim R_p(Y)=2$ and $\dim R_p(P_C) =4$,
$P_C$ is not contained in $Y$.
Choose a general section $f \in H^0(\P^N,I_{Y/\P^N} \otimes \calo_{\P^N}(2))$,
and let $D \subset \P^N$ be the corresponding hypersurface of degree $2$.
By the condition i), the generality of $f$, and $P_C \not \subset Y $,
we have $P_C \not \subset D$ .
Thus as an effective divisor on $P_C$,
$C $ is contained in $ D | _{P_C}$.
Since $C$ is a conic and $D | _{P_C} $ is an effective divisor of degree $2$ on $P_C \cong \P^2$,
$C$ and $D | _{P_C} $ coincide as schemes.
This means $ I_{C/\P^N} = I_{D/ \P^N} + I_{P_C / \P^N}$.
Thus it holds that
\begin{align*}
I_{C / Y } =  I_{C/\P^N} /  I_{Y/\P^N} &= ( I_{D/ \P^N} + I_{P_C / \P^N} ) /  I_{Y/\P^N} \\
&= ( I_{Y/ \P^N} + I_{P_C / \P^N} ) /  I_{Y/\P^N} .
\end{align*}
The last equality follows from $I_{D/ \P^N} \subset I_{Y/ \P^N} $.
Since
\[
H^0(\P^N,I_{P_C / \P^N}  \otimes \calo_{\P^N}(1)) \otimes \calo_{\P^N} \arw I_{P_C / \P^N}  \otimes \calo_{\P^N}(1)
\]
is surjective
and $H^0(\P^N,I_{P_C / \P^N}  \otimes \calo_{\P^N}(1))=H^0(\P^N,I_{C/\P^N} \otimes \calo_{\P^N}(1))$,
this claim follows.
\end{proof}
\noindent
\textbf{Step 2.} 
Let $(C,H) \in I$ be a general element,
and set $X=Y \cap H$.
In this step, we show that $X$ is smooth.
(Note that we do not know $H$ is general in $|\calo_{\P^N}(1) \otimes \m_p|$ a priori.
Thus we have to check the smoothness of $X$.) 
It is clear that $X \setminus C$ is smooth by Claim \ref{glob_gen_of_ideal} and the generality of $(C,H)$.
Hence it is enough to show that $X$ is smooth along $C$.

Consider
\[
B := \{ (q,H ) \in C \times |\calo_{\P^N}(1) \otimes I_{C/\P^N} | \, | \, Y \cap H \text{ is singular at } q \}.
\]
For each $q \in C$,
the fiber of the projection $B \arw C$ over $q$ corresponds to the kernel of
\begin{align*}
H^0(\P^N,I_{C/\P^N} \otimes \calo_{\P^N}(1))  &\arw H^0(Y,I_{C/Y} \otimes \calo_Y(1)) \\
&\arw (I_{C/Y} + \m_{Y,q}^2) / \m_{Y,q}^2,
\end{align*}
where $\m_{Y,q}$ is the maximal ideal sheaf on $Y$ at $q$.
By Claim \ref{glob_gen_of_ideal},
this map is surjective.
Hence the projection $B \arw C$ is a $\P^k$-bundle
for $k=\dim |\calo_{\P^N}(1) \otimes I_{C/\P^N} | - \codim (C,Y)$.
Thus we have
\[
\dim B=k+1=\dim |\calo_{\P^N}(1) \otimes I_{C/\P^N} | +2 -\dim Y.
\]
If $\dim Y \geq 3$,
the natural projection $B \arw |\calo_{\P^N}(1) \otimes I_{C/\P^N} | $ is not generically surjective
since $\dim B < \dim |\calo_{\P^N}(1) \otimes I_{C/\P^N} |$.
This means $X=Y \cap H$ is smooth along $C$ for general $H \in |\calo_{\P^N}(1) \otimes I_{C/\P^N} |$.
When $\dim Y=2$,
it holds $K_Y.C=-4$ because $\dim \mathcal R=2  $ and $C$ is a free rational curve.
Thus we have $C^2=2$.
By the Hodge index theorem,
$Y \subset \P^N$ is a quadric surface and $C \sim \calo_Y(1)$.
Hence $X$ must coincide with $C$,
which is smooth.
\vspace{2mm}
\ \\
\textbf{Step 3.}
Let $(C,H) \in I$ and $X=Y \cap H$ be as in Step 2.
To prove $\dim \pi_2^{-1} (H) =0$,
it is enough to show $N_{C/X} \cong \calo_{C}^{\oplus n-1}$ for $n=\dim X$
because $\dim \pi_2^{-1} (H) \leq h^0(C, N_{C/X} \otimes \m_p)$.
Write
\[
f^* N_{C/Y}=\bigoplus_{i=1}^{n} \calo_{\P^1}(a_i)
\]
for integers $a_1 \geq \ldots \geq a_{n} $,
where $f$ is an isomorphism $\P^1 \arw C$.
(We use $f$ not to confuse $\calo_{\P^N}(1) |_C$ and the degree $1$ invertible sheaf on $C$.) 
Since $C$ is free on $Y$,
it follows that $a_i \geq 0$ for any $i$.
Furthermore, $\sum_i a_i=\dim \mathcal R=2$ holds since $C$ is free.
Hence we have
\[
f^* N_{C/Y} = \calo_{\P^1}(2) \oplus \calo^{\oplus n-1} \text{ or } \ \calo_{\P^1}(1)^{\oplus 2} \oplus \calo^{\oplus n-2} .
\]

\vspace{1mm}
Let $\alpha \in H^0(\P^N,I_{C/\P^N} \otimes \calo_{\P^N}(1)) $
be a section corresponding to $H$.
From the natural surjection
\[
I_{C/Y} \otimes \calo_Y(1) \arw I_{C/Y} / I_{C/Y}^2 \otimes \calo_Y(1) |_C \cong N_{C/Y}^{\vee} \otimes \calo_Y(1) |_C
\]
and Claim \ref{glob_gen_of_ideal},
we obtain a surjective map
\[
\delta: H^0(\P^N,I_{C/\P^N} \otimes \calo_{\P^N}(1)) \otimes \calo_C \arw N_{C/Y}^{\vee} \otimes \calo_Y(1) |_C.
\]
Furthermore, there exist natural isomorphisms
\[
N_{C/Y}^{\vee} \otimes \calo_Y(1) |_C  \cong  \mathcal{H}om_C (N_{C/Y}, \calo_Y(1) |_C)
\cong  \mathcal{H}om_C (N_{C/Y}, N_{X/Y} |_C).
\]
The image $\delta(\alpha) \in H^0(C, N_{C/Y}^{\vee} \otimes \calo_Y(1) |_C) \cong \Hom_C (N_{C/Y}, N_{X/Y} |_C)$
of $\alpha$ induces an exact sequence
\[
0 \arw N_{C/X} \arw N_{C/Y} \stackrel{\delta(\alpha)}{\arw} N_{X/Y} |_C \arw 0 .
\]
Since
$f^* N_{C/Y} =  \calo_{\P^1}(2) \oplus \calo^{\oplus n-1} $ or $\calo_{\P^1}(1)^{\oplus 2} \oplus \calo^{\oplus n-2} $
and $f^* N_{X/Y} |_C =\calo_{\P^1}(2)$,
we have $N_{C/X} \cong \calo_C^{\oplus n-1} $
if the restriction of $\delta(\alpha) $ on $ \calo(2) $ or $\calo(1)^{\oplus 2} $ is surjective.
Since $\alpha$ is general,
this follows from the subjectivity of $\delta$.
\end{proof}

\subsection{Proof of Theorem \ref{intro_thm1}}\label{subsection_proof_of_cor}

As a corollary of Theorem \ref{main_thm} and Proposition \ref{existence_of_conic},
we obtain Theorem \ref{intro_thm1}.
As stated in Remark \ref{rem for thm 1},
Theorem \ref{intro_thm1} can be rephrased as follows.

\begin{cor}[=Theorem \ref{intro_thm1}]\label{rational_homog}
Let $Y \subset \P^N$ be a rational homogeneous space of Picard number $1$,
which is embedded by the ample generator.
Let $X$ be a complete intersection variety in $Y$ of hypersurfaces of degrees $d_1 \leq \ldots \leq d_r$
such that $-K_X$ is ample.
For $p \in X$,
it holds that
\begin{align*}
\ep (X,\calo_X(1); p) = \left\{ 
\begin{array}{cl} 
 1 &  \text{when }F_p(X)\ne \emptyset,\\
d_r/(d_r-1) & \text{when } F_p(X)= \emptyset \text{ and } d_r \geq 2 ,\\ 
2 & \text{when } F_p(X)= \emptyset \text{ and } d_r =1.\\
\end{array} \right.
\end{align*}
\end{cor}

\begin{proof}
By the adjunction formula,
$-K_X$ is ample if and only if $\sum_{j=1}^r d_j \leq i(Y)-1 $,
where $i(Y)$ is the positive integer satisfying $-K_Y=\calo_Y( i(Y) )$,
i.e., the Fano index of $Y$.
It is well known that $\dim F_p(Y) = i(Y) -2$.
For instance, the existence of lines is proved in \cite[Theorem V.1.1.15]{Ko},
and $\dim F_p(Y) $ is computed by the deformation theory.
By Lemma \ref{cone of F_p(X)},
it is easy to show that $X$ is covered by lines
if $\sum_{j=1}^r d_j < i(Y)-1 $
as in the first paragraph of the proof of Theorem \ref{main_thm}.

Furthermore,
$Y \subset \P^N$ is cut out by quadrics,
i.e.,
$d_p(Y) \leq 2$ holds for any $p \in Y$ (see \cite{Li} for example).
Thus we can apply Theorem \ref{main_thm} if $d_r \geq 2$,
and this corollary follows in that case.

\vspace{1.5mm}
Assume $F_p(X) = \emptyset$ and $d_r=1$.
In this case,
$i(Y)=\sum_{j=1}^r d_j +1=r+1$ and $Y$ is not a projective space.
Since $d_p(X) \leq d_p(Y)=2$,
$\ep (X,\calo_X(1); p)  \geq 2$ follows from Theorem \ref{intro thm2}.

To show the opposite inequality,
we use Proposition \ref{existence_of_conic}.
By the lower semicontinuity of Seshadri constants,
we may assume $X=Y \cap \bigcap_{j=1}^r H_j$ for general hyperplanes $H_j \subset \P^N$
and $p \in X$ is a general point.
Set
\[
Y'=Y \cap \bigcap_{j=1}^{r-1} H_j ,
\]
and let us check that $Y'$ satisfies the conditions i) and ii)
in Proposition \ref{existence_of_conic}.

Since $Y$ is cut out by quadrics,
so is $Y'$.
Hence i) holds for $Y'$.
By Lemma \ref{cone of F_p(X)},
$F_p(Y')$ is an intersection of $F_p(Y)$ and general $r-1$ hyperplanes in $F_p(\P^N)$.
Thus
$F_p(Y')$ is a non-empty $0$-dimensional set because $\dim F_p(Y)=i(Y)-2=r-1$.
Furthermore $F_p(Y) $ is not an irreducible linear space in $F_p(\P^N) \cong \P^{N-1}$ by \cite[Proposition 5]{Hw}.
Thus we have $\# F_p(Y') \geq 2$.
Since any line in $F_p(Y')$ is free by the generality of $p$,
the union $l \cup l'$ of two lines $l \not =l' \in F_p(Y')$
is smoothable into a free conic $C$ on $Y'$ (cf.\ \cite[Proposition 4.24]{Deb}).
Thus $Y'$ is covered by conics.
Since $-K_{Y'}=\calo_{Y'}(2)$,
we have
\[
\dim R_p(Y') = -K_{Y'}.C-2=2
\]
for $C \in R_p(Y')$,
which is nothing but ii).
Thus we can apply Proposition \ref{existence_of_conic}
to $Y'$,
and we have a conic on $X=Y' \cap H_r$ passing through $p$.
This means $\ep (X,\calo_X(1); p)  \leq 2 $ and the proof is finished.
\end{proof}

\end{document}